\documentclass[12pt, letterpaper,reqno]{amsart}
\usepackage{amsthm,amssymb,mathrsfs}
\usepackage{amsmath,amscd}
\usepackage{graphicx}
\usepackage{xcolor}
\usepackage{pinlabel}
\usepackage{fourier}
\usepackage[breaklinks,colorlinks,citecolor=blue,linkcolor=blue,
 urlcolor=teal]{hyperref}

\theoremstyle{plain}
\newtheorem{theorem}[equation]{Theorem} 

\newtheorem{lemma}[equation]{Lemma}

\theoremstyle{definition}

\renewcommand{\leq}{\leqslant}
\renewcommand{\geq}{\geqslant}
\renewcommand{\epsilon}{\varepsilon}
\renewcommand{\phi}{\varphi}

\renewcommand{\preceq}{\preccurlyeq}

\DeclareMathOperator{\area}{Area}

\numberwithin{equation}{section}

\makeatletter
  \def\tagform@#1{\maketag@@@{%
   \textbf{(\ignorespaces#1\unskip\@@italiccorr)}}}%
   \renewcommand{\eqref}[1]{\textup{\maketag@@@{(\ignorespaces%
        {\ref{#1}}\unskip\@@italiccorr)}}}
\makeatother

\begin{document}

\title[Richard Thompson's group  $T$]{The Dehn function of Richard Thompson's group  $T$}
\author{Kun Wang}
\address{Mathematics Department\\
        Henan University\\
        Kaifeng, 475000, P. R. China\\
        }

\author{Zhu-Jun Zheng}
\address{Mathematics Department\\
       South China University of Technology\\
        Guangzhou 510641\\
        P. R. China}
\author{Junhuai Zhang}

\email{zhengzj@scut.edu.cn}

\begin{abstract}
We improve Guba's result about the Dehn function of R.Thompson's group $T$ and get that $\Phi_T(n)\preceq n^5$, where $\Phi_T$ is the Dehn function of group $T$.
\end{abstract}

\thanks{The second author acknowledges support from NSFC(11571119) and NSFC(11475178).}

\maketitle

\thispagestyle{empty}

\section{introduction}

\bigskip
By $F_s$ we denote the free group on $a_1,a_2,a_3,\ldots,a_s$. Let $N$ be the normal closure of the defining relators $R_1,R_2,\ldots,R_t$ in the free group $F_s$ and $H$ be the quotient group $F_s/N$. For any word $w\in N$, its area is defined by the minimal number
$$\area(w)=\min\{k|w=\prod_{i=1}^k{U_i^{-1}R_{j_i}U_i}\},$$
where $U_i\in F_s,$ $R_{j_i}\in \{R_1,R_2,\ldots,R_t\}^{\pm1}$. Using van Kampen's Lemma \cite{9,10}, the area of $w$ is equal to the smallest number of cells in a van Kampen diagram if its boundary label is $w$. If two words $v,w$ are equal in $H$ then we denote by $\parallel v=w\parallel$ the number $\area(vw^{-1})$. The Dehn function $\Phi$ is the map from $\mathbb{N}$ to itself defined as
$$\Phi(n)=\max\{\area(w)|w\in N,|w|\leq n\},$$
where $|w|$ is the length of the word $w$.

The R.Thompson's group $F$ was discovered by Richard Thompson in the 1960s. The group $F$ is torsion-free, and it can be defined by the following presentation
\begin{equation*}
  \left\langle{x_0,x_1,x_2,\ldots \mid x_j^{x_i}=x_{j+1}(i<j)}\right\rangle,
\end{equation*}
where $x_j^{x_i}=x_i^{-1}x_jx_i$. $F$ is finitely presented, it can be given by
$$\left\langle{x_0,x_1,\mid x_2^{x_1}=x_3,x_3^{x_1}=x_4}\right\rangle,$$
where $x_n=x_0^{-(n-1)}x_1x_0^{n-1}$ for any $n\geq2$.
For Richard Thompson's group $T$, it has the following finite presentation:
\begin{equation*}
\left\langle{x_0,x_1,c_1,\mid x_2^{x_1}=x_3,x_3^{x_1}=x_4,c_1=x_1c_2,c_2x_2=x_1c_3,c_1x_0=c_2^2,c_1^3=1}\right\rangle,
\end{equation*}
where $x_n=x_0^{-(n-1)}x_1x_0^{n-1},c_n=x_0^{-(n-1)}c_1x_1^{n-1}$ for any $n\geq2$.

Let $f$, $g$ be two non-decreasing functions from $\mathbb{N}$ to itself. $f\preceq g$ means that there exists a positive integer constant $C$ such that $f(n)\leq Cg(Cn)+Cn$ for all $n$. $f\sim g$ if and only if $f\preceq g$ and $g\preceq f$.

For a finitely presented group there exists a unique Dehn function, modulo the equivalence relation $\sim$. The group has solvable word problem if and only if its Dehn function has s recursive upper bound \cite{12,13,14}. A finitely presented group is word-hyperbolic group if and only if its Dehn function $\sim n$ \cite{6}.

In 1997, Guba and Sapir \cite{7} proved that the Dehn function $\Phi_F(n)$ of group $F$ is strictly subexponential, namely $\Phi_F(n)\preceq n^{\log n}=2^{\log^2n}$. After one year, Guba \cite{11} proved that the Dehn function of $F$ satisfies $\Phi_F(n)\preceq n^5$. Finally, Guba \cite{1} proved that $\Phi_F(n)\sim n^2$ in 2006.

For R.Thompson's group $T$, Guba \cite{2} proved that $\Phi_T(n)\preceq n^7$. We improve this result as the following:

\begin{theorem}\label{t}
\emph{The Dehn function of R.Thompson's group} $T$ \emph{satisfies polynomial upper bound and} $\Phi_T(n)\preceq n^5$.
\end{theorem}
\bigskip

\bigskip 
\section{preliminaries}
\bigskip
In this section, we improve some inequalities by using van Kampen diagrams which shall be used to estimate the Dehn function of $T$ in next section.

\begin{lemma}\label{1}
\emph{For any} $0<k\leq n$,

(1)
\begin{equation}\label{L11}
\parallel c_n=x_nc_{n+1}\parallel=1;
\end{equation}

(2)
\begin{equation}\label{L12}
\parallel c_nx_k=x_{k-1}c_{n+1}\parallel=O(n^2);
\end{equation}

(3)
\begin{equation}\label{L13}
\parallel c_nx_0=c_{n+1}^2\parallel=O(n^3).
\end{equation}
\end{lemma}

\begin{proof}
For (\ref{L11}), we just use the relation $c_1=x_1c_2$ one time,
$$c_n=x_0^{-(n-1)}c_1x_1^{n-1}=x_0^{-(n-1)}x_1c_2x_1^{n-1}=x_0^{-(n-1)}x_1x_0^{n-1}x_0^{-(n-1)}c_2x_1^{n-1}=x_nc_{n+1}.$$

Now let us prove (\ref{L12}). Let $h(n,k)=\parallel c_nx_k=x_{k-1}c_{n+1}\parallel$. It is easy to see that $h(n,1)=0$, $h(2,2)=1$. At first, we prove the case $n\geq3$, $k=2$. Consider the following van Kampen diagram:
\begin{center}
\includegraphics{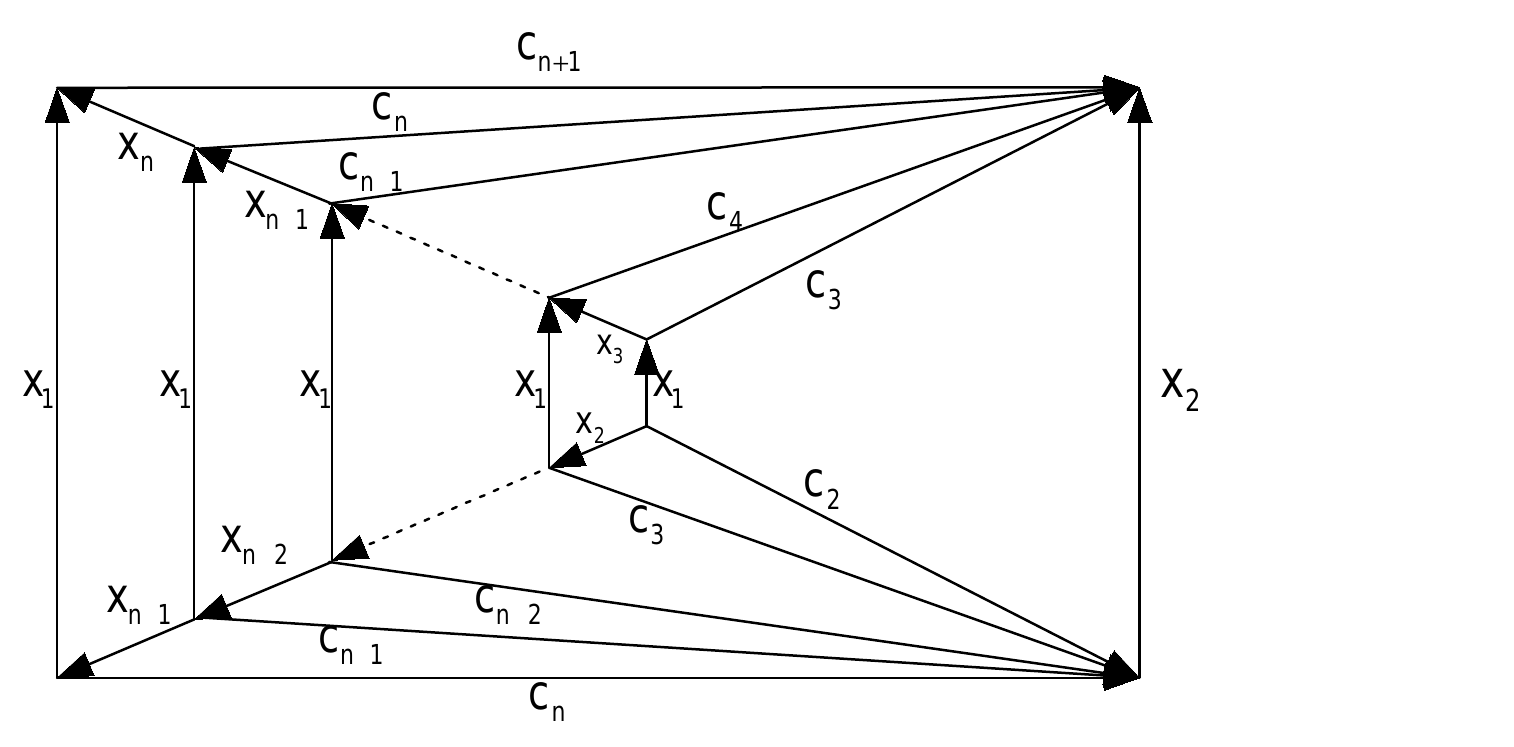}
\end{center}
applying (\ref{L11}) the above diagram shows that
\begin{equation*}
\begin{split}
h(n,2)
&\leq h(2,2)+n-2+n-2+\parallel x_2x_3\cdots x_{n-2}x_{n-1}x_1=x_1x_3\cdots x_{n-1}x_n\parallel\\
&=h(2,2)+2n-4+\parallel x_0^{-1}x_1x_0x_0^{-2}x_1x_0^2\cdots x_1x_0^{n-3}x_0^{-(n-2)}x_1x_0^{n-2}x_1\parallel.
\end{split}
\end{equation*}
By the result of Guba's which says the Dehn function of R.Thompson group $F$ is quadratic. We have

$\parallel x_2x_3\cdots x_{n-2}x_{n-1}x_1=x_1x_3\cdots x_{n-1}x_n\parallel$

$=\parallel x_0^{-1}x_1x_0^{-1}x_1x_0^{-1}\cdots x_1x_0^{-1}x_1x_0^{-1}x_1x_0^{n-2}x_1=x_1x_0^{-2}x_1x_0^{-1}x_1x_0^{-1}\cdots x_1x_0^{-1}x_1x_0^{n-1}\parallel$

$=O(n^2).$

So we get $h(n,2)=O(n^2)$. For $k\geq3$, we consider the following diagram
\begin{center}
\includegraphics{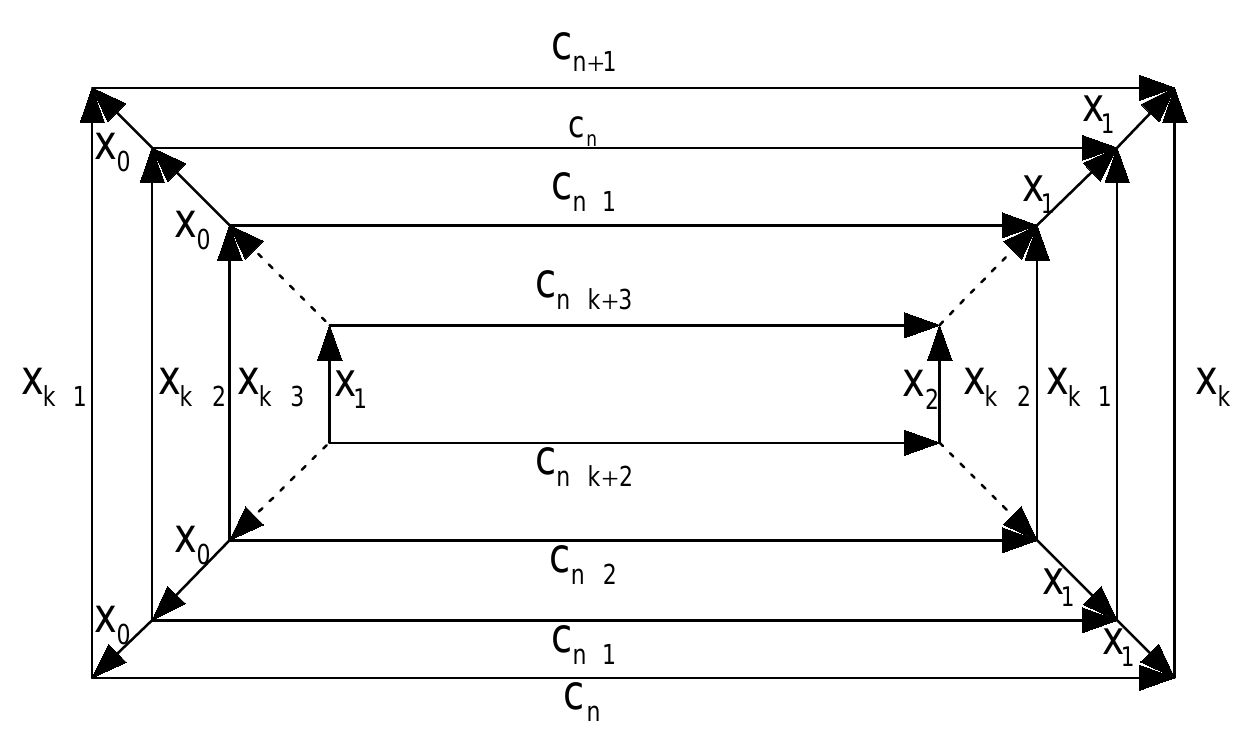}
\end{center}
which shows that $h(n,k)\leq h(n-k+2,2)+\parallel x_1^{-(k-1)}x_2x_1^{k-1}=x_k\parallel=O(n^2)$ by Guba's result.
So we proved (\ref{L12}).

For (\ref{L13}), let $h(n)=\parallel c_nx_0=c_{n+1}^2\parallel$. Clearly, $h(1)=\parallel c_1x_0=c_2^2\parallel=1$. Now let $n\geq2$, consider the following diagram
\begin{center}
\includegraphics{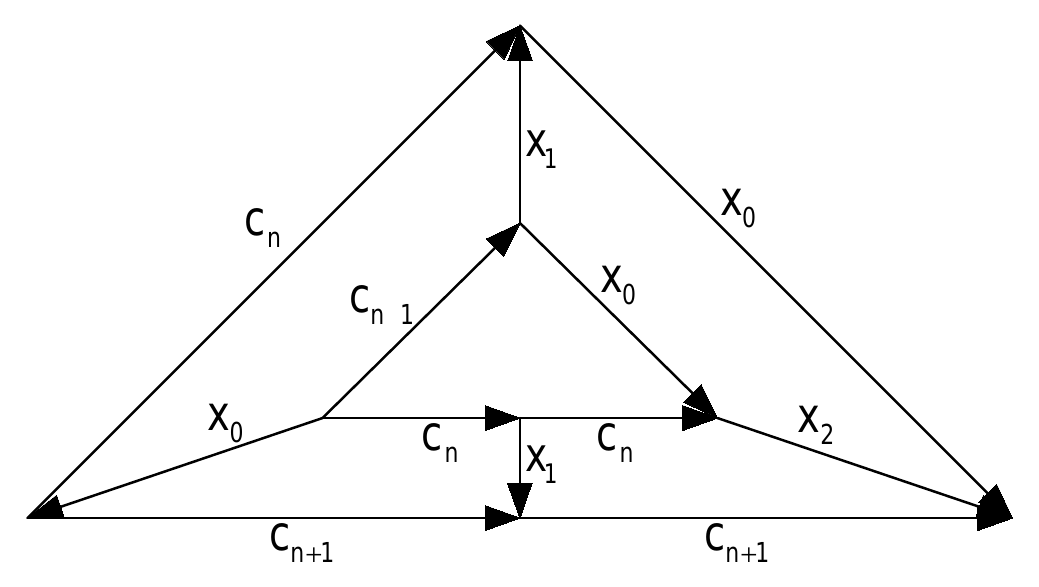}
\end{center}
shows that $h(n)\leq h(n-1)+h(n,2)$. So we have that $h(n)=O(n^3)$, which obtains (\ref{L13}).
\end{proof}

\begin{lemma}\label{2}
\emph{Let} $n\geq1$\emph{,} $1\leq m\leq n+1$, $0\leq r,s\leq n$ \emph{be integers.}

(1)
\begin{equation}\label{L21}
c_n^mx_r=
  \left\{
   \begin{aligned}
   &x_{r-m}c_{n+1}^m,r\geq m,\\
   &c_{n+1}^{m+1}, r=m-1,\\
   &x_{r+n+2-m}c_{n+1}^{m+1}, r<m-1;\\
   \end{aligned}
   \right.
\end{equation}

(2)
\begin{equation}\label{L22}
x_s^{-1}c_n^m=\left\{
   \begin{aligned}
   &c_{n+1}^{m+1}x_{s+m-n-2}^{-1},s\geq n+2-m,\\
   &c_{n+1}^m, s=n+1-m,\\
   &c_{n+1}^m, s\leq n-m;\\
   \end{aligned}
   \right.
\end{equation}

(3)
\begin{equation}\label{L23}
c_n^m=x_{n-m+1}c_{n+1}^m;
\end{equation}

(4)
\begin{equation}\label{L24}
c_n^m=c_{n+1}^{m+1}x_{m-1}^{-1};
\end{equation}

(5)
\begin{equation}\label{L25}
c_n^{n+2}=1.
\end{equation}

\emph{Each of the equalities} (\ref{L21}) (\ref{L22}) (\ref{L23}) (\ref{L24}) \emph{needs} $O(n^3)$ \emph{elementary steps and} (\ref{L25}) \emph{needs} $O(n^4)$ \emph{steps}.
\end{lemma}

\begin{proof}
For (\ref{L21}), in case $r\geq m$ and consider the following diagram
\begin{center}
\includegraphics{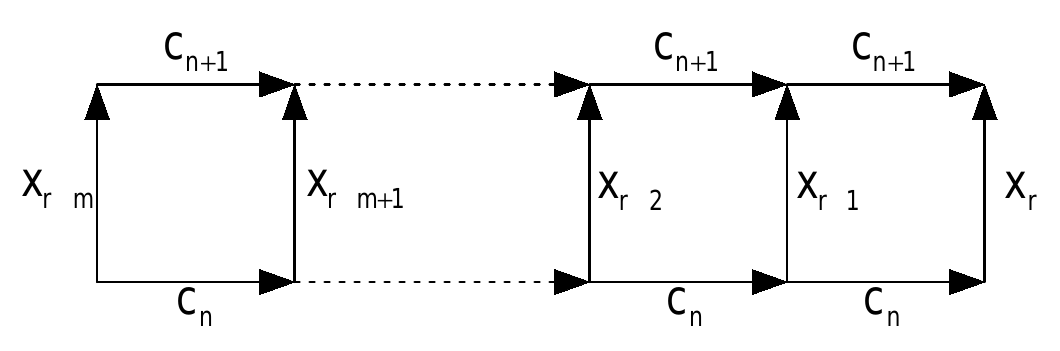}
\end{center}
the word is $c_n^m$ labelled by the bottom path, we see that $\parallel c_n^mx_r=x_{r-m}c_{n+1}^m\parallel=O(n^3)$.

In case $r=m-1$. If $r=0$, then $\parallel c_nx_0=c_{n+1}^2\parallel=O(n^3)$. If $r>0$, then the following diagram

\begin{center}
\includegraphics{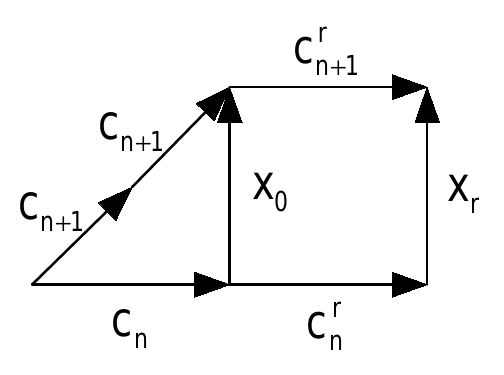}
\end{center}
shows that $\parallel c_n^mx_r=c_{n+1}^{m+1}\parallel=O(n^3)+O(n^3)=O(n^3)$.
For (\ref{L23}),  Lemma \ref{1} and the following diagram

\begin{center}
\includegraphics{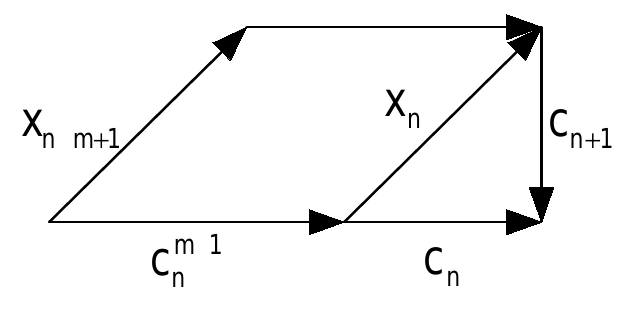}
\end{center}
shows that $\parallel c_n^m=x_{n-m+1}c_{n+1}^m\parallel=O(n^3)$.

Now we return to prove (\ref{L21}), in case $r<m-1$. We consider the following diagram and apply (\ref{L23}) and (\ref{L21})

\begin{center}
\includegraphics{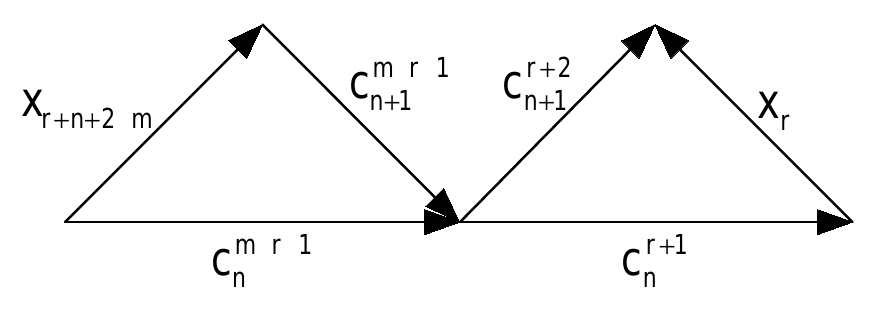}
\end{center}
we have $\parallel c_n^mx_r=x_{r+n+2-m}c_{n+1}^{m+1}\parallel=O(n^3)$.

For (\ref{L22}) and (\ref{L24}), they can be obtained from (\ref{L21}) and (\ref{L23}). In addition, they all need $O(n^3)$ elementary steps.

Finally, let us prove (\ref{L25}). We construct the following diagram
\begin{center}
\includegraphics{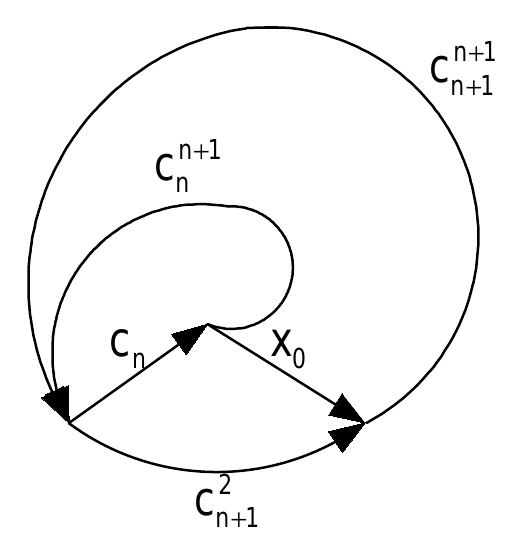}
\end{center}
so $\parallel c_{n+1}^{n+3}=1\parallel\leq\parallel c_n^{n+2}=1\parallel+O(n^3)$, which implies $\parallel c_n^{n+2}=1\parallel=O(n^4)$.
\end{proof}

The definition of the complexity $\xi$ which will be occured in the following lemma can be found in \cite{2} and we improve the corresponding lemma in it.

\begin{lemma}\label{3}
\emph{Let} $v$ \emph{be a positive word}, $n=\xi(v)$. \emph{Then there exist a positive word} $w$ \emph{and an integer} $k(1\leq k\leq n+1)$ \emph{such that} $c_1v$ \emph{equals} $wc_{n+1}^k$ \emph{in} $T$ \emph{and} $\parallel c_1v=wc_{n+1}^k\parallel=O(n^4)$, \emph{where} $\xi(w)\leq\xi(v)$.
\end{lemma}

\begin{proof}
We prove the lemma by induction on $|v|$, which $|v|$ is the length of the word $v$. It is obviously in the case of $|v|=0$. Now let $|v|>0$. Then we can write $v=v'x_j$ for some $j\geq0$. By the induction assumption, set $n'=\xi(v')$, we have a positive word $w'$, an integer $k'(1\leq k'\leq n'+1)$, such that $\xi(w')\leq n'$ and $\parallel c_1v'=w'c_{n'+1}^{k'}\parallel=O(n^4)$. Consider several cases.

(1) $j>n'+1$. Applying Lemma \ref{2} we have
\begin{equation*}
\begin{split}
c_1v &=c_1v'x_j \\
&=w'c_{n'+1}^{k'}x_j\\
&=w'x_{n'+2-k'}c_{n'+2}^{k'}\\
&=\cdots=w'x_{n'+2-k'}\cdots x_{n-k'}c_n^{k'}x_n\\
&=w'x_{n'+2-k'}\cdots x_{n-k'}x_{n-k'}c_{n+1}^{k'}.
\end{split}
\end{equation*}
We take $w=w'x_{n'+2-k'}\cdots x_{n-k'}x_{n-k'}$. Then we can obtain $\parallel c_1v=wc_{n+1}^{k'}\parallel=O(n^4)+O(n^3)+O(n^4)=O(n^4)$. We now compute the complexity of $w$
\begin{equation*}
\begin{split}
\xi(w)
&=\max\{n-k',\xi(w'x_{n'+2-k'}\cdots x_{n-k'})+1\}\\
&=\max\{n-k'+1,\xi(w'x_{n'+2-k'}\cdots x_{n-k'-1})+2\}\\
&=\max\{n-k'+1,\xi(w'x_{n'+2-k'}\cdots x_{n-k'-2})+3\}\\
&=\cdots=\max\{n-k'+1,\xi(w'x_{n'+2-k'})+n-n'-1\}\\
&=\max\{n-k'+1,\xi(w')+n-n'\}\\
&=n.
\end{split}
\end{equation*}

(2) $j\leq n'+1$. In this case, $n=\xi(v)=\xi(v'x_j)=\max\{j,n'+1\}=n'+1$.

If $j\geq k'$, we take $w=w'x_{j-k'}$, $k=k'$.

If $j=k'-1$, we take $w=w'$, $k=k'+1$.

If $j<k'-1$, then we take $w=w'x_{j+n'+3-k'}$, $k=k'+1$.

It is easy to verify by using Lemma \ref{2} in all these cases that $\xi(w)\leq n$ and $\parallel c_1v=wc_{n+1}^k\parallel=O(n^4)$.
\end{proof}

\bigskip
\section{Proof of theorem \ref{t}}

\bigskip

Let $w$ be a word of length $n$ over the alphabet $\{x_0^{\pm1}x_1^{\pm1},c_1\}$, then we claim that there exists positive words $p,q$ and an integer $m(0\leq m\leq n+2)$ such that $pc_{n+1}^mq^{-1}$ equals $w$ in $T$,
$\xi(p),\xi(q)\leq n$ and $\parallel w=pc_{n+1}^mq^{-1}\parallel=O(n^5)$. If $n=0$, the claim is trivial. Let $n>0$. We write $w=yw'$, where $y=x_s^{\pm1}(s=0,1)$ or $y=c_1$. For $w'$ by the inductive assumption, there exist positive words $p',q'$ of complexity $\leq n-1$ and an integer $m'(0\leq m'\leq n+1)$ such that $w'=p'c_n^{m'}(q')^{-1}$ in $T$.

(1) $y=x_s(s=0,1)$. Let $p=x_sp'$. After a simple calculation we have $\xi(p)=\max\{\xi(p'),|p'|+1\}$. If $m'=0$, then $m=0$, $q=q'$. If $m'\geq1$, then by Lemma \ref{2} $c_n^{m'}=c_{n+1}^{m'+1}x_{m'-1}^{-1}$. In this case, we set $m=m'+1$, $q=q'x_{m'-1}$. $\xi(q)=\max\{m'-1,\xi(q')+1\}\leq n$. And we have the following inequality
$$\parallel w=pc_{n+1}^mq^{-1}\parallel \leq\parallel w'=p'c_n^{m'}(q')^{-1}\parallel+O(n^3).$$

(2)$y=x_s^{-1}(s=0,1)$. By \cite{1}and the lemma in \cite{2},there exist a positive word $p$, integers $t,\delta=0,1$ such that $\parallel x_s^{-1}p'=px_t^{-\delta}\parallel=O(n^2)$, $t\leq |p'|+s$, $\xi(p)\leq\xi(p')+1$. If $m'=0$, we just take $q=q'x_t^\delta$ and $m=0$. Now let $m'>0$. If $\delta=0$, then take $q=q'x_{m'-1}$, $m=m'+1$. So clearly by Lemma \ref{2}, $w=x_s^{-1}w'=x_s^{-1}p'c_n^{m'}(q')^{-1}=qc_{n+1}^{m'+1}x_{m'-1}^{-1}(q')^{-1}=pc_{n+1}^mq^{-1}$. If $\delta=1$.

If $t\geq n+2-m'$, then we take $m=m'+1$, $q=q'x_{t+m'-n-2}$.

If $t=n+1-m'$, then take $m=m'$,$q=q'$.

If $\leq n-m'$, then we take $m=m'$, $q=q'x_{t+m'}$. It is easy to calculate the complexity of $p,q$ $\leq n$ and we have the following inequality
$$\parallel w=pc_{n+1}^mq^{-1}\parallel \leq\parallel w'=p'c_n^{m'}(q')^{-1}\parallel+O(n^3).$$

(3) $y=c_1$. In this case, it is similar to be obtained from \cite{2}, but we have the inequality more precisely $$\parallel w=pc_{n+1}^mq^{-1}\parallel \leq\parallel w'=p'c_n^{m'}(q')^{-1}\parallel+O(n^4).$$
Thus we have proved the claim by induction. Now it is not difficult to get $\parallel w=1\parallel=O(n^5)$, where $w$ is any word of length $\leq n$ on the alphabet $\{x_0^{\pm1},x_1^{\pm1},c_1^{\pm1}\}$ and $w=1$ in $T$. \qed

\end{document}